\documentclass[Afour,sageh,times]{my-sagej}

\usepackage{moreverb,url}
\usepackage{amssymb}
\usepackage{amsfonts}
\usepackage{amsmath}
\usepackage{amsthm}
\usepackage[overload]{empheq}
\usepackage{relsize}
\usepackage[colorlinks,bookmarksopen,bookmarksnumbered,citecolor=blue,urlcolor=blue]{hyperref}

\newtheorem{theorem}{Theorem}
\newtheorem{lemma}{Lemma}

\begin{document}

\runninghead{M.~S.~Ali, M.~Shamsi, H.~Khosravian-Arab, D.~F.~M.~Torres and F.~Bozorgnia}

\def\journalname{Journal of Vibration and Control}

\title{A space-time pseudospectral discretization method for solving 
diffusion optimal control problems\\ 
with two-sided fractional derivatives}

\author{Mushtaq Salh Ali\affilnum{1},
Mostafa Shamsi\affilnum{1},
Hassan Khosravian-Arab\affilnum{1},
Delfim F. M. Torres \affilnum{2} and
Farid Bozorgnia\affilnum{3}}

\affiliation{\affilnum{1}Department of Applied Mathematics, 
Faculty of Mathematics and Computer Science, 
Amirkabir University of Technology, Tehran, Iran\\
\affilnum{2}Department of Mathematics, Center for Research 
and Development in Mathematics and Applications (CIDMA),
University of Aveiro, 3810-193 Aveiro, Portugal\\
\affilnum{3}Department of Mathematics, 
Instituto Superior T\'{e}cnico, 
1049-001 Lisbon, Portugal}

\corrauth{Mostafa Shamsi,
Department of Applied Mathematics, 
Faculty of Mathematics and Computer Science, 
Amirkabir University of Technology, 
Hafez Avenue no.~424, 
Tehran, 
Iran}

\email{m\_shamsi@aut.ac.ir}


\begin{abstract}
We propose a direct numerical method for the solution of an optimal control 
problem governed by a two-side space-fractional diffusion equation.  The 
presented method contains two main steps. In the first step, the space 
variable is discretized by using the Jacobi--Gauss pseudospectral discretization 
and, in this way, the original problem is transformed into a classical integer-order 
optimal control problem. The main challenge, which we faced in this step, 
is to derive the left and right fractional differentiation matrices. 
In this respect, novel techniques for derivation of these matrices are presented.  
In the second step, the Legendre--Gauss--Radau pseudospectral method is employed. 
With these two steps, the original problem is converted into a convex quadratic 
optimization problem, which can be solved efficiently by available methods. 
Our approach can be easily implemented and extended to cover fractional optimal control 
problems with state constraints. Five test examples are provided to demonstrate 
the efficiency and validity of the presented method. The results show that 
our method reaches the solutions with good accuracy and a low CPU time.
\end{abstract}



\maketitle


\section{Introduction}

Fractional (non-integer order) order diffusion equations, 
in comparison with classical integer order counterparts,  
can describe more accurately irregular diffusion processes, 
such as gas diffusion in fractal porous media and  heat conduction 
\citep{Wu2015335-NEW,Wu20176054-NEW,CAMYang2017,CAMYang2017b}. 
Accordingly, the numerical solution of fractional diffusion 
equations has gained considerable attention 
\citep{MEERSCHAERT200680,Doha20141,Zaky2016,Yang2017,CAMYang2016,Feng2018441-NEW,Chen2018-NEW,Yang2017863-NEW}.

Optimal control of fractional diffusion equations has recently received special attention, 
due to their applications in various fields, such as control of temperature in a thermo 
conduction process or mass diffusive transport in a porous media \citep{Ning}.

\citet{MR2739436} considered  the optimal control of 
a time-fractional diffusion equation in a bounded domain and 
investigated the questions of existence and uniqueness of solution. 
Moreover, the first-order optimality conditions are derived in \citep{MR2739436}. 

The optimal control of time-fractional diffusion equations 
with state constraints and non-homogeneous Dirichlet boundary conditions 
is considered in \citet{MR2824729} and \citet{MR2824734}, respectively. In these papers, 
existence and uniqueness of solution and first-order optimality conditions are also studied.
Regarding numerical methods for solving optimal control problems 
of time-fractional diffusion equations, we refer to \cite{doi:10.1080/00207160.2017.1417591,MR3739910}.
The reader interested on the optimal control of fractional partial differential equations is referred to
\citet{MR3787674,MR3654793,Zaky20182667,Bai2018338,Darehmiraki20182149}
and references therein.

Recently, an optimal control problem governed by a two-sided space-fractional 
diffusion equation was considered in \citet{Ning} and \citet{doi:10.1177/1077546317705557}. 
More precisely, a constrained time-dependent optimal control problem defined 
on the space interval $(0,1)$ and time period $[0, T]$ was considered. 
The aim is to find the control $u(x,t)\in L^2\left(0,T;(0,1)\right)$ 
and state $y(x,t)\in L^2\left(0,T;(0,1)\right)$
such that the space-fractional diffusion equation 
\begin{subequations}
\label{main-prob}
\begin{equation*}
\frac{\partial y}{\partial t}(x,t)=c(x,t)\left[r\, {}_{\text{\tiny 0}}
\mathcal D_{x}^{2-\beta} y(x,t) +(1-r) {}_{x}
\hspace{-1pt}\mathcal D_{\text{\tiny 1} }^{2-\beta} 
y(x,t)\right]
\end{equation*}
\vspace{-10pt}
\begin{equation}
+f(x,t)+u(x,t),\label{main-prob-dyn}
\end{equation}
with initial condition 
\begin{equation}
\label{main-prob-ini}
y(x,0)=g(x),\quad x\in (0,1),
\end{equation}
and the boundary conditions 
\begin{equation}
\label{main-prob-bc}
y(0,t)=y(1,t)=0,\quad t\in(0,T),
\end{equation}
are satisfied and the following performance index is minimized:
\begin{align}
J[u]:=&\tfrac{1}{2}\int_{0}^{T}\int_{0}^{1}\big[y(x,t)-z(x,t)\big]^2 \text{d}x \text{d}t 
\nonumber\\
&\qquad+\tfrac{1}{2}\int_{0}^{T}\int_{0}^{1} u^2(x,t) \text{d}x \text{d}t.\label{main-prob-obj}
\end{align}
Moreover, the following constraint on the control function is considered:
\begin{equation}
\label{mainp-path}
u(x,t)\ge u_{\min}(x,t),\quad x\in (0,1),\ t\in(0,T),
\end{equation}
where $u_{\min}$ is a known function.
\end{subequations}
In problem \eqref{main-prob},  $z(x,t)$ is a known function, which represents 
the observed or desired values of state function $y$, $f(x,t)$ is the source or sink term, 
$d(x,t)$ is the diffusivity coefficient and $2-\beta$ is the order of  diffusion, where $0<\beta<1$.
Moreover, ${}_{\text{\tiny 0}}\mathcal D_{x}^{2-\beta}$ and 
${}_{x}\hspace{-1pt}\mathcal D_{\text{\tiny 1} }^{2-\beta}$ denote the left and right 
Riemann--Liouville  fractional derivatives, defined (see, e.g., \citet{Malinowska2012}) as
\begin{align*}
&{}_{\text{\tiny 0}}\mathcal D_{x}^{\alpha} g(x) 
=\dfrac{1}{\Gamma(2-\alpha) }\dfrac{\text{d}^{2}}{\text{d}x^{2}} 
\int_{0}^{x}{(x-s)}^{1-\alpha } g(s) \text{d}s.\\
&{}_{x}\hspace{-1pt}\mathcal D_{\text{\tiny 1} }^{\alpha} g(x) 
= \dfrac{1}{\Gamma(2-\alpha) } \dfrac{\text{d}^{2}}{\text{d}x^{2}}
\int_{s}^{1}{(s-x)}^{1-\alpha} g(s) \text{d}s.
\end{align*}

Generally speaking, the two main challenges in solving problem \eqref{main-prob} numerically 
are: (i) the problem contains both left and right fractional derivatives,
(ii) the fractional derivatives are nonlocal and singular operators.
Moreover, when using finite-difference-based methods, the stiffness 
of the approximation matrix of the fractional derivative is added 
to the two aforementioned drawbacks \citep{doi:10.1177/1077546317705557}.  
Thus, finite difference schemes, for solving problem \eqref{main-prob}, 
require expensive computation time and storage cost.  
In this regard, the main concern in the finite difference based methods for solving the problem \eqref{main-prob} 
is to reduce the demand for time and memory computation \citep{Ning,doi:10.1177/1077546317705557}. 
In \citet{Ning}, necessary optimality conditions for problem \eqref{main-prob} are derived 
and a faithful and fast gradient projection method is developed to solve them numerically. 
In \citet{doi:10.1177/1077546317705557}, to get more reduction of computational times,  
a parallel-in-time algorithm for implementing the gradient projection method 
is presented for problem \eqref{main-prob}.

In this paper, we present another method for solving problem \eqref{main-prob} numerically.
There are two major differences between the method we propose here and the ones of 
\citet{Ning} and \citet{doi:10.1177/1077546317705557}. First, our approach is based on
high-order and global pseudospectral methods rather than finite difference schemes.
Second, our method is direct, that is, does not rely on necessary optimality conditions. 
In a direct method, the optimal control problem is solved by transcribing it into a 
Non-Linear Programming (NLP) problem; 
thereafter, an NLP-solver is used to solve the resulting NLP problem
\citep{MR3089375,BELLOSALATI2018,Behroozifar20182494,Mashayekhi20181621}. 
The direct methods are easily implemented and inequality  
constraints on state and control are handled simpler, in comparison 
with indirect methods \citep{MR3089375,BELLOSALATI2018,Mohammadzadeh2018}.

Pseudospectral methods approximate the unknown function(s) using  
interpolating polynomials with specific collocation points such as 
Legendre--Gauss (LG) \citep{Benson06JoGCaD},
Legendre--Gauss--Lobatto (LGL) \citep{Elnagar95ACITo} 
and Legendre--Gauss--Radau (LGR) points \citep{Garg10A}. 
These selections lead to the three most common types 
of pseudospectral methods, which are referred as 
LG pseudospectral \citep{Benson06JoGCaD}, 
LGL pseudospectral \citep{Fahroo01JoGCaDa,Elnagar95ACITo}, 
and LGR pseudospectral  \citep{Garg11,Garg10A} methods.

It is a well-known fact that, to solve ordinary or partial differential equations 
with a simple domain and a smooth solution, pseudospectral methods can usually 
achieve an accurate solution even with a small number of nodes. On the other hand, 
for problems with a non-smooth solution, the accuracy of pseudospectral methods is reduced. 
However, for problems with a non-smooth solution, they lead to a reasonably accurate 
approximation with less demand on computational time and computer memory.

Due to the mentioned computational efficiency, pseudospectral methods have been popular 
for the numerical solution of optimal control problems governed by integer order differential equations 
\citep{Garg10A,paper_F1,Ezz-Eldien201716,doi:10.1080/00207179.2017.1399216}. 
More recently, pseudospectral methods have been extended for solving optimal control problems governed 
by fractional ordinary and partial differential equations  
\citep{Tang2017333,doi:10.1080/00207160.2017.1417591,MMA:MMA4366}.

Here we present a direct method for solving problem \eqref{main-prob} that consists in two steps.
In the first step, the Jacobi--Gauss pseudospectral method is used for space discretization 
and, as a result, the problem is reduced to a classical optimal control problem.
It is worthwhile to note that, because of the existence of both left and right space-fractional 
derivatives in the considered problem, we need to compute the left and right 
fractional differentiation matrices.  These differentiation matrices,  
which are approximations of the left and right fractional operators, 
play an important role in the pseudospectral method and, therefore, 
need to be obtained accurately. In this paper, by using some useful properties 
of the Jacobi polynomials, we present efficient strategies for computing 
the left and right fractional differentiation matrices. 
In the second step of our method, we apply the Legendre--Gauss--Radau 
pseudospectral method. The resulting ordinary optimal control problem 
is then transcribed into a quadratic optimization problem.
In addition, for easy implementation and analysis, the standard form 
of the quadratic optimization problem is derived.

The rest of the paper is organized as follows. In section 
``Preliminaries and notations'', 
we recall some necessary definitions and relevant properties of the 
Jacobi polynomials, quadrature rules, and the Kronecker product. 
Next section is devoted to the first step of our method, 
where the Jacobi--Gauss pseudospectral method is applied for space 
discretization of the problem. Moreover, the computation of the
differentiation matrices is presented in this section.  
Then the second step of our method is presented, 
where the Legendre--Gauss--Radau pseudospectral method is utilized 
to reduce the problem to a quadratic optimization one. 
In section ``Numerical experiments'', five illustrative examples are investigated 
to assess the accuracy and efficiency of the proposed numerical method.
We end with a section of conclusions.


\section{Preliminaries and notations}
\label{sec:2}

Jacobi polynomials have great flexibility to develop efficient numerical methods 
for solving a wide range of fractional differential models. Consequently, in the last decade, 
Jacobi polynomials have been widely used to solve fractional problems 
\citep{ESMAEILI20113646,ESMAEILI2011918,BHRAWY2015876,Dehghan20161547,BHRAWY2016832,Zaky2018}. 
Our method uses Jacobi polynomials too. Thus, in this section we briefly review 
the Jacobi polynomials, Jacobi quadrature rules, and some relevant theorems 
on the derivatives of Jacobi polynomials. 


\subsection{Jacobi Polynomials}

The Jacobi polynomials $P_n^{(a,b)}(\tau)$, where $\tau \in [-1,1]$ and 
$n=0,1,\ldots$, are given explicitly by \citet{Gautschi}:
\begin{equation*}
P_n^{(a,b)}(\tau)=\tfrac 1 {2^{n}}\sum_{k=0}^{n}
\begin{pmatrix}
n+a \\
n-k 
\end{pmatrix}
\begin{pmatrix}
n+b \\
k 
\end{pmatrix}
(\tau-1)^k(\tau+1)^{n-k}.
\end{equation*}
However, in practice, one can use the so-called recurrence Bonnet's relation 
to generate the Jacobi polynomials in a stable and accurate manner \citep{Gautschi}.
If $a,b>-1$, Jacobi polynomials are called \emph{classical Jacobi polynomials}.
The well-known Legendre polynomials are a special case of the Jacobi polynomials when $a=b=0$.
In the following, some useful properties of the classical Jacobi polynomials are reviewed.

The classical Jacobi polynomials are orthogonal on the canonical
interval $[-1,1]$ with respect to the weight function
$(1-\tau)^a(1+\tau)^b$, i.e.,
\begin{align}
\int_{-1}^1 (1-\tau)^a&(1+\tau)^b P_n^{(a,b)}(\tau) 
P_m^{ (a,b)}(\tau) d\tau\nonumber \\
&=
\begin{cases}
0, &  m\ne n, \\
\frac{2^{a+b+1}\Gamma(n+a+1)\Gamma(n+b+1)}{n!(2n+a+b+1)
\Gamma(n+a+b+1)}, 
& m=n.
\end{cases}
\label{orthoP}
\end{align}
A useful formula that relates the Jacobi polynomials and their
derivatives is
\begin{align*}
&\frac {d^k}{d\tau^k} P_n^{(a,b)}(\tau)=\frac{\Gamma(k+n+a+b+1)}{2^{k}
\Gamma(n+a+b+1)}P_{n-k}^{(a+k,b+k)}(\tau),\\
&\quad\qquad n\ge k.
\end{align*}
The Jacobi polynomials also satisfy the following properties:
\begin{align}
	P_{k}^{(a, b-1)}(2x - 1) = &\tfrac{k + a + b}{2k + a + b}
P_{k}^{(a, b)}(2x - 1) \nonumber\\
&\ \ +  \tfrac{k + a}{2k + a + b}
P_{k-1}^{(a, b)}(2x - 1),\label{Jac_Prop_1}\\
	P_k^{(a  - 1,b )}(2x-1) = &\tfrac{k + a  + b }{2k+a+b}P_k^{(a ,b )}(2x-1) \nonumber \\
&\ \ - \tfrac{k + b}{2k+a+b}P_{k - 1}^{(a ,b )}(2x-1).\label{Jac_Prop_2}
\end{align}
In the following, we recall two important theorems, 
which are used later to establish our method.

\begin{theorem}[See \citet{zayernouri2013fractional}]
\label{Zayer13_1}
Let $r-\alpha>-1$ and $q+\alpha>-1$. Then, for $x\in[0,1]$, we have:
\begin{align*}
&{}_{\text{\tiny 0}}\mathcal D_{x}^{\alpha}\left[x^{r}P_{k}^{(q,r)}\left({2x}-1\right)\right]\\
&\qquad\quad=\tfrac{\Gamma(k+r+1)}{\Gamma(k+r-\alpha+1)}x^{r-\alpha}
P_{k}^{(q+\alpha,r-\alpha)}\left({2x}-1\right).
\end{align*}		
\end{theorem}  

\begin{theorem}[See  \citet{zayernouri2013fractional}]
\label{Zayer13_2}
Let $q-\alpha>-1$ and $r+\alpha>-1$. Then, for $x\in[0,1]$, we have:
\begin{align*}
&{}_{x}\hspace{-1pt}\mathcal D_{\text{\tiny 1}}^{\alpha}
\left[(1-\tau)^{q}P_{k}^{(q,r)}\left({2x}-1\right)\right]\\
&\quad\qquad = \tfrac{\Gamma(k+q+1)}{\Gamma(k+q-\alpha+1)}(1-\tau)^{r-\alpha}
P_{k}^{(q-\alpha,r+\alpha)}\left({2x}-1\right).
\end{align*}		
\end{theorem}  


\subsection{Jacobi nodes and quadratures} 

In the pseudospectral methods, the Jacobi--Gauss and Jacobi--Gauss--Radau 
nodes are successfully used as discretization points \citep{Garg10A}. Here 
we recall the definitions of these nodes and their corresponding  
quadrature rules  \citep{Gautschi,Garg11}. 

Since the Jacobi polynomials are orthogonal in $[-1,1]$, all the zeros of 
$P_n^{(a,b)}(\tau)$ are simple and belong to the interval $(-1,1)$ \citep{Gautschi}. 
These zeros are called the Jacobi--Gauss nodes with parameters $a$ and $b$, 
which we denote by $\{\xi^{(a,b)}_i\}_{i=1}^n$. The Jacobi--Gauss quadrature rule 
with parameters $a$ and $b$ is based on the Jacobi--Gauss nodes 
$\{\xi^{(a,b)}_i\}_{i=1}^n$ and can be used for
approximating the integral of a function over 
the interval $[-1,1]$ with weight $(1-x)^a(1+x)^b$ as
\begin{equation} 
\label{JGQ}
\int_{-1}^{1} (1-x)^a(1+x)^b f(x) \text{d}x 
\simeq \sum_{i=1}^n \omega^{(a,b)}_i f(\xi^{(a,b)}_i),
\end{equation}
where $\omega^{(a,b)}_i$, $i=1,\ldots,n$, are the Jacobi--Gauss 
quadrature weights. The above quadrature is exact whenever $f(x)$ 
is a polynomial of degree equal or less than $2n-1$, i.e., 
in the Jacobi--Gauss quadrature rule, the degree of exactness is $2n-1$.

In the Jacobi--Gauss--Radau nodes $\{\tau^{(a,b)}_i\}_{i=1}^m$, 
the first node is $\tau_1=-1$ and the last $m-1$ nodes are the zeros 
of $P^{(a,b)}_{m-2}(t) + P^{(a,b)}_{m-1}(t)$. We note that the zeros 
of $P^{(a,b)}_{m-2}(t) + P^{(a,b)}_{m-1}(t)$ belong to $(-1,1)$. Thus, 
the Jacobi--Gauss--Radau points lie in the interval $[-1,1)$.  
The Jacobi--Gauss--Radau quadrature rule with parameters $a$ and $b$ is given by
\begin{equation} 
\label{JGLQ}
\int_{-1}^{1} (1-t)^a(1+t)^b f(t) \text{d}t 
\simeq \sum_{i=1}^m \varpi^{(a,b)}_i f(\tau^{(a,b)}_i),
\end{equation}
where $\tau^{(a,b)}_i$, $i=1,\ldots,m$, are the Jacobi--Gauss--Lobatto 
quadrature weights. The degree of exactness of the Jacobi--Gauss--Radau 
quadrature is $2m-2$.


\subsection{Kronecker products}

Let $\mathbf A\in \mathbb{R}^{m\times n}$, $\mathbf B\in
\mathbb{R}^{p\times q}$. The Kronecker product of matrices 
$\mathbf A$ and $\mathbf B$ is defined as 
\begin{equation*}
\mathbf A \otimes \mathbf B=\left[
\begin{array}{ccc}
{a}_{11}\mathbf B & \cdots & a_{1n}\mathbf B \\
\vdots & \ddots & \vdots \\
{a}_{m1}\mathbf B & \cdots & a_{mn}\mathbf B \\
\end{array}
\right] \in \mathbb{R}^{mp\times nq}.
\end{equation*}

\begin{theorem}[See \citet{laub2005matrix}]
\label{KRON}
For any three matrices $\mathbf A$, $\mathbf B$ and $\mathbf E$, 
for which the matrix product $\mathbf {ABE}$ is defined, one has 
\begin{equation*}
\text{\bf{vec}}(\mathbf {ABE}) 
= (\mathbf R^T\otimes \mathbf A)\text{\bf{vec}}(\mathbf B),
\end{equation*}
where $\text{\bf{vec}}$ is the vectorization operator, 
which converts a matrix into a column vector by
stacking the columns of the matrix on the top of one another.
\end{theorem}


\section{Step I: The Jacobi--Gauss pseudospectral method for spatial discretization}
\label{sec:3}

The first step of our method consists to discretize the spatial variable $x$.
For that, we use the pseudospectral method based on the Jacobi--Gauss nodes 
with parameters $(a,b)=(1,1)$.  Let $n$ be a positive integer number. Since 
the spatial domain of the problem is $[0,1]$, we consider the  
Jacobi--Gauss nodes correspondent to the interval $[0,1]$ as
\begin{equation}
\label{ShiftedNodes}
\hat \xi_i:=\frac 1 2 \left(\xi^{\text{\tiny(1,1)}}_i+1\right),
\quad i=1,\dots,n.
\end{equation}
The control and state functions are approximated as
\begin{align}
&u(x,t)\simeq \tilde u_n(x,t)
:=\sum_{j=1}^n u_j(t)\ell_j(x),\label{ap_u_n}\\
&y(x,t)\simeq \tilde y_n(x,t)
:=\sum_{j=1}^{n} y_j(t)\mathring{\ell}_j(x),\label{ap_y_n}
\end{align}
where $\ell_j(x)$, $j=1,\ldots,n$, are the Lagrange basis 
polynomials based on the Jacobi--Gauss nodes $\left\{ \hat\xi_i \right\}_{i=1}^n$, i.e.,
\begin{equation}
\label{elldef}
\ell_j(x):=\prod_{\substack{k=1 \\ k\ne j}}^{n} 
\frac {x- \hat\xi_k }{\hat\xi_j - \hat\xi_k }
\end{equation}
and
\begin{equation*}
\mathring{\ell}_j(x):=\frac{x(1-x)}{\hat\xi_j
\left(1-\hat\xi_j\right)}\ell_j(x).
\end{equation*}
Note that $\ell_j(x)$ and $\mathring{\ell}_j(x)$ satisfy 
the following  Kronecker properties in the Jacobi--Gauss nodes:
\begin{equation}
\label{Kronpp}
\mathring{\ell}_{j}(\hat{\xi}_i )=\ell_{j}(\hat{\xi}_i )
=\begin{cases}
0 \ \ if  \ \ i\neq j\\
1 \ \ if \ \  i= j
\end{cases},\ \ j=1,\ldots,n.
\end{equation}
Moreover, $\mathring{\ell}_j(0)=\mathring{\ell}_j(1)=0$, $j=1,\ldots,n$. 
Therefore, for each value of $y_j(t)$, $j=1,\ldots,n$, the approximation 
$\tilde y_n(x,t)$ satisfies the boundary conditions \eqref{main-prob-bc}. 

By substituting the approximations \eqref{ap_u_n} and \eqref{ap_y_n} 
in equation \eqref{main-prob-dyn}, and by simple algebraic manipulation, 
we get 
\begin{align*}
&\sum_{j=1}^{n} \dot y_j(t)\mathring{\ell}_j(x)
= c(x,t)\left[r \sum_{j=1}^{n} y_j(t)  {}_{\text{\tiny 0}}
\mathcal D_{x}^{2-\beta} \mathring{\ell}_j(x) \right.\\ 
&\left. +(1-r)\sum_{j=1}^{n} y_j(t) {}_{x}
\hspace{-1pt}\mathcal D_{\text{\tiny 1} }^{2-\beta}\mathring{\ell}_j(x)\right]  
+ f(x,t) +\sum_{j=1}^n u_j(t)\ell_j(x).
\end{align*}
Now, by collocating the above equation in $\hat{\xi}_i$, $i=1,\ldots,n$, 
and using the Kronecker property \eqref{Kronpp}, we get
\begin{multline}
\label{finp}
\dot y_i(t)= c(\hat{\xi}_i,t)
\left[r \sum_{j=1}^{n} y_j(t)  {}_{\text{\tiny 0}}
\mathcal D_{x}^{2-\beta} \mathring{\ell}_j(\hat{\xi}_i) 
\right.\\
\left. \hspace{-3pt}+(1-r)\sum_{j=1}^{n} y_j(t) {}_{x}\hspace{-1pt}\mathcal 
D_{\text{\tiny 1} }^{2-\beta}\mathring{\ell}_j(\hat{\xi}_i)\right]
 + f(\hat{\xi}_i,t) + u_i(t).
\end{multline}
Let 
\begin{align}
 d_{ij}^+:=  {}_{\text{\tiny 0}}\mathcal D_{x}^{2-\beta} 
\mathring{\ell}_j(\hat{\xi}_i),\quad
 d_{ij}^-:= {}_{x}\hspace{-1pt}\mathcal 
D_{\text{\tiny 1} }^{2-\beta} \mathring{\ell}_j(\hat{\xi}_i).\label{lrd}
\end{align}
Then, equation \eqref{finp} can be written as
\begin{align*}
\dot y_i(t)= 
c(\hat{\xi}_i,t)\left[r \sum_{j=1}^{n} d_{ij}^+\,  y_j(t) 
+(1-r)\sum_{j=1}^{n} d_{ij}^-\, y_j(t) \right]
\end{align*}
\vspace{-10pt}
\begin{align}
\qquad\qquad + f(\hat{\xi}_i,t) + u_i(t),\qquad  i=1,\ldots,n. \label{SODE}
\end{align}
The above equations form a linear system of time-varying 
ordinary differential equations. To derive the vector form 
of the above system, let
\begin{align*}
&\mathbf y(t):=\left[ y_1(t),\dots, y_n(t)\right]^{\text{T}},\
\mathbf u(t):=\left[ u_1(t),\dots, u_n(t)\right]^{\text{T}},\\
&\mathbf f(t):=\big[f(\hat{\xi}_1,t),\dots,f(\hat{\xi}_n,t)\big]^{\text{T}},\\
&\mathbf C(t):=\text{diag}\Big(c(\hat{\xi}_1,t),\dots,c(\hat{\xi}_n,t)\Big).
\end{align*}
Using the above notations, the system of differential equations \eqref{SODE} 
can be expressed as 
\begin{equation*}
\mathbf{\dot y}(t)=\mathbf C(t) \left(r\, {\mathlarger{\mathbf{\mathfrak D}}}_{+}^{2-\beta} 
+(1-r) {\mathlarger{\mathbf{\mathfrak D}}}_{-}^{2-\beta} \right) \mathbf y(t) 
+ \mathbf f(t)+ \mathbf u(t),
\end{equation*}
where 
${\mathlarger{\mathbf{\mathfrak D}}}_{+}^{2-\beta}$ 
and ${\mathlarger{\mathbf{\mathfrak D}}}_{-}^{2-\beta}$
are called the left and right fractional differentiation matrices, 
respectively, and are defined as
\begin{align*}
{\mathlarger{\mathbf{\mathfrak D}}}_+^{2-\beta}:=
\Big[
d_{ij}^+
\Big]_{\substack{i=1,\dots,n\\j=1,\dots,n}},\quad
{\mathlarger{\mathbf{\mathfrak D}}}_{-}^{2-\beta}:=
\Big[
d_{ij}^-
\Big]_{\substack{i=1,\dots,n\\j=1,\dots,n}}.
\end{align*}
The initial condition \eqref{main-prob-ini}
is discretized to 
\begin{equation}
\mathbf y(0)=\mathbf g:=\left[g(\hat{\xi}_1),
\ldots,g(\hat{\xi}_n)\right]^{\text{T}}.
\end{equation}
Moreover, the inequality condition \eqref{mainp-path} on the control 
function is discretized into the following inequality constraints:
\begin{equation*}
u(\hat{\xi}_i,t)\ge u_{\min}(\hat{\xi}_i,t),
\quad i=1,\ldots,n.
\end{equation*}
The above inequalities can be rewritten in vector form as
\begin{equation}
\label{pathineq}
\mathbf u(t)\ge \mathbf u_{\min}(t),
\end{equation}
where 
\begin{equation*}
\mathbf u_{\min}(t):=\left[ u_{\min}(\hat{\xi}_1,t),
\dots, u_{\min}(\hat{\xi}_n,t)\right]^{\text{T}}.
\end{equation*}

Now, we turn to discretize the performance index. By using 
the Jacobi--Gauss quadrature rule \eqref{JGQ}, we approximate 
the inner integrals in the performance index \eqref{main-prob-obj} as
\begin{align*}
J(u)\simeq J_n[u]=\tfrac{1}{2}&\int_{0}^{T}\sum_{j=1}^{n}\tfrac 1 2 
\omega^{\text{\tiny(1,1)}}_i\big[y(\hat\xi_j,t)-z(\hat\xi_j,t)\big]^2  
\text{d}t \\
&+\tfrac{1}{2}\int_{0}^{T}\sum_{j=1}^{n} \tfrac 1 2 
\omega^{\text{\tiny(1,1)}}_j u^2(\hat{\xi}_j,t) \text{d}t,
\end{align*}
where $\omega^{\text{\tiny(1,1)}}_j$, $j=1,\dots,n$, 
are the Jacobi--Gauss weights. We rewrite $J_n$ as 
\begin{align*}
J_n[\mathbf u]=\tfrac{1}{2}&\int_{0}^{T}\big[\mathbf y(t)
-\mathbf z(t)\big]^{\text{T}}\mathbf W \big[\mathbf y(t)
-\mathbf z(t)\big] \text{d}t \nonumber\\
&\quad + \tfrac{1}{2}\int_{0}^{T}\mathbf u^\text{T}(t) \mathbf W  
\mathbf u(t)\,\text{d}t,
\end{align*}
where
\begin{align*}
&\mathbf z(t):=
\begin{bmatrix}\frac{z(\hat{\xi}_1,t)}{\hat{\xi}_1\left(1-\hat{\xi}_1\right)}
&\dots&\frac{z(\hat{\xi}_n,t)}{\hat{\xi}_n\left(1-\hat{\xi}_n\right)}
\end{bmatrix}^{\text{T}},\\
&\mathbf W:=\frac 1 2\text{diag}\Big(  
\omega^{\text{\tiny(1,1)}}_1,\dots,\omega^{\text{\tiny(1,1)}}_n \Big).		
\end{align*}
In summary, by the Jacobi--Gauss pseudospectral spatial discretization, 
problem \eqref{main-prob} is transcribed into the following 
classical optimal control problem: 
\begin{subequations}
\label{main-OOCP}
\begin{align}[left = \empheqlbrace\,]
&\min J_n=\tfrac 1 2\hspace{-2pt}\int_{0}^{T}\hspace{-5pt}\big[\mathbf y(t)-\mathbf z(t)\big]^{\text{T}}\mathbf W \big[\mathbf 
y(t)-\mathbf z(t)\big] \text{d}t \nonumber\\ 
&\qquad\quad\qquad\quad+\tfrac{1}{2}\int_{0}^{T}\mathbf u^\text{T}(t) \mathbf W  \mathbf 
u(t)\,\text{d}t, \label{objooocp}\\
&\ \ \text{s.t.}\qquad\mathbf{\dot y}(t)=\mathbf C(t)  
\, {\mathlarger{\mathbf{\mathfrak D}}}_{\pm}^{2-\beta}\, \mathbf y(t) 
+ \mathbf f(t)+ \mathbf u(t),\label{dynooocp}\\
&\qquad\qquad \mathbf y(0)=\mathbf g,\label{iniooocp}\\
&\qquad\qquad \mathbf u(t)\ge \mathbf u_{\min}(t),\label{pathooocp}
\end{align}
\end{subequations}
where
\begin{equation*}
{\mathlarger{\mathbf{\mathfrak D}}}_{\pm}^{2-\beta}
:=r {\mathlarger{\mathbf{\mathfrak D}}}_{+}^{2-\beta} 
+(1-r) {\mathlarger{\mathbf{\mathfrak D}}}_{-}^{2-\beta}. 
\end{equation*}


\subsection*{On the derivation of the left and right fractional differentiation matrices}

As we have seen, the fractional differentiation matrices 
${\mathlarger{\mathbf{\mathfrak D}}}_+^{2-\beta}$ 
and
$ {\mathlarger{\mathbf{\mathfrak D}}}_{-}^{2-\beta}$ 
have a crucial role in the proposed Jacobi--Gauss pseudospectral method. 
These matrices simplify the discretization process by replacing 
left and right fractional differentiations with matrix-vector products. 
At first glance, it seems that by using \eqref{lrd}, the $(i,j)$-th component 
of the differentiation matrix ${\mathlarger{\mathbf{\mathfrak D}}}_+^{2-\beta}$  
/  $ {\mathlarger{\mathbf{\mathfrak D}}}_{-}^{2-\beta}$ can be easily computed 
by taking the analytical left/right fractional derivative of $\mathring{\ell}_j(x)$ 
and evaluating it at collocation points $\hat{\xi}_i$. However, taking the analytical 
fractional derivative, especially for large $n$, is not practical and accessible. 
Accordingly, we need stable and accurate methods for generating these matrices. 
In this respect, in what follows we present some lemmas and theorems, 
which play key roles in computing the left and right fractional differentiation 
matrices in an accurate and stable method.

In the sequel, we use $\hat P^{(a,b)}_k(x)$, $k=0,1,\ldots$, 
to denote the shifted Jacobi polynomials, which are defined as
\begin{equation*}
\hat P^{(a,b)}_k(x)=P_{k}^{(a,b)}(2x-1),\quad k=0,1,\dots.
\end{equation*}  

\begin{lemma}
\label{Lemm_1}
For $j=1,\dots,n$, we have
\begin{equation}
\mathring{\ell}_j(x)=\sum_{k=1}^n\lambda_{jk} \,x(1-x)\hat P^{(1,1)}_{k-1}(x),
\end{equation}
where
\begin{equation}
\label{Eq_0}
\lambda_{jk}:=\tfrac{(2k+3)(k+2)}{k+1}\tfrac{\omega^{\text{\tiny{\rm(1,1)}}}_j}
{\hat{\xi}_j\left(1-\hat{\xi}_j\right)}
\ \hat P_{k-1}^{(1,1)}(\hat{\xi}_i).
\end{equation}
\end{lemma}

\begin{proof}
We first note that $\ell_j(x)$, defined in \eqref{elldef}, is a polynomial of degree $n-1$.
Thus, we can expand it in terms of the shifted Jacobi polynomials 
$\hat P_{k-1}^{(1,1)} (x)$, $k=1,\dots,n$, as follows:
\begin{equation}
\label{e2829}
\ell_j(x)=\sum_{k=1}^n\hat\lambda_{j,k}\hat P_{k-1}^{(1,1)}(x).
\end{equation}
Multiplying both sides of the above equation by $x(1-x)\hat P_{k-1}^{(1,1)}(x)$, 
then integrating both sides of the resulted equation from $0$ to $1$ and 	
using the orthogonality property \eqref{orthoP}, we finally get	that
\begin{equation}
\label{Eq_1}
\hat \lambda_{jk}:=\tfrac{(2k+3)(k+2)}{k+1}\int_0^1 x(1-x) 
\ell_j(x) \hat P_{k-1}^{(1,1)}(x)\,\text{d}t.
\end{equation}
Now, by noting that $\ell_j(x) \hat P_{k-1}^{(1,1)}(x)$ is a polynomial 
of degree at most $2n-2$, we can compute exactly the integral in the 
right-hand side of \eqref{Eq_1} by using the Jacobi--Gauss quadrature 
rule \eqref{JGQ} with $a=b=1$. In this way, we have
\begin{equation*}
\hat \lambda_{jk}:=\tfrac{(2k+3)(k+2)}{k+1}\sum_{i=1}^n 
\omega^{\text{\tiny(1,1)}}_i \ell_j(\hat{\xi}_i) 
\hat P_{k-1}^{(1,1)}(\hat{\xi}_i).
\end{equation*}	
Finally, using the Kronecker property \eqref{Kronpp},  
$\hat\lambda_{jk}$ is obtained as
\begin{equation*}
\hat \lambda_{jk}:=\tfrac{(2k+3)(k+2)}{k+1}\omega^{\text{\tiny(1,1)}}_j	
\ \hat P_{k-1}^{(1,1)}(\hat{\xi}_i).
\end{equation*}
The proof is complete by multiplying both sides of equation \eqref{e2829}  
by $\frac{x(1-x)}{\hat{\xi}_j(1-\hat{\xi}_j)}$.
\end{proof}

\begin{theorem}
\label{Thm_1}
The $(i,j)$-th element of the left fractional differentiation matrix 
${\mathlarger{\mathbf{\mathfrak D}}}_+^{2-\beta}$ 
can be obtained explicitly as
\begin{equation*}
d_{ij}^+:= {}_{\text{\tiny 0}}\mathcal D_{x}^{2-\beta}\mathring{\ell}_j(\hat{\xi}_i)
=\sum_{k=1}^n\lambda_{jk} \, \zeta_{k}(\hat{\xi}_i),
\end{equation*}
where $\hat{\xi}_i$ are the shifted Jacobi--Gauss nodes 
defined in \eqref{ShiftedNodes} and 
\begin{align*}
&\zeta_{k}(x):=\tfrac{\Gamma(k+1)}{\Gamma(k-1+\beta)}x^{\beta-1}\hat
P_{k-1}^{(3-\beta,\beta-1)}(x)\\
&-\tfrac{\Gamma(k+2)}{\Gamma(k+\beta)}
\tfrac{k+2}{2k+1}x^{\beta}\hat P_{k-1}^{(3-\beta,\beta)}\hspace{-2pt}(x) \hspace{-2pt}-\hspace{-2pt} \tfrac{\Gamma(k+1)}{\Gamma(k-1+\beta)}
\tfrac{k}{2k+1}x^{\beta}\hat P_{k-2}^{(3-\beta,\beta)}\hspace{-2pt}(x).
\end{align*}
\end{theorem}

\begin{proof}
Using Lemma~\ref{Lemm_1}, we have 
\begin{equation}
\hspace{-5pt}{}_{\text{\tiny 0}}\mathcal D_{x}^{2-\beta}\hat \ell_j(x)
\hspace{-2pt}= \hspace{-2pt}\sum_{k=1}^n \hspace{-3pt}\lambda_{jk} \,{}_{\text{\tiny 0}}\mathcal 
D_{x}^{2-\beta}  \hspace{-3pt}\left[x(1-x)P_{k-1}^{(1,1)}(2x-1) \hspace{-2pt}\right] \hspace{-3pt}.\label{TEMP32}
\end{equation}
With $a=1$ and $b=2$, the property \eqref{Jac_Prop_1} gives
\begin{align*}
&x(1-x)\hat P_{k-1}^{(1,1)}(x)
=x\hat P_{k-1}^{(1,1)}(x)-x^2\hat P_{k-1}^{(1,1)}(x)\\
&\ \ = x\hat P_{k-1}^{(1,1)}(x)-\tfrac{k+2}{2k+1}x^2\hat P_{k-1}^{(1,2)}(x)
-\tfrac{k}{2k+1}x^2\hat P_{k-2}^{(1,2)}(x).
\end{align*}
Using the above equation and Theorem~\ref{Zayer13_1}, we get
\begin{align*}
&\qquad \quad {}_{\text{\tiny 0}}\mathcal D_{x}^{2-\beta} \left[x(1-x)\hat P_{k-1}^{(1,1)}(x)\right]\\
&=\tfrac{\Gamma(k+1)}{\Gamma(k-1+\beta)}x^{\beta-1}\hat P_{k-1}^{(3-\beta,\beta-1)}(x)\hspace{-2pt}-\hspace{-2pt}\tfrac{\Gamma(k+2)}{\Gamma(k+\beta)}\tfrac{k+2}{2k+1}x^{\beta}\hat P_{k-1}^{(3-\beta,\beta)}(x)
\\
&\qquad\ \ -\tfrac{\Gamma(k+1)}{\Gamma(k-1+\beta)}\tfrac{k}{2k+1}x^{\beta}\hat P_{k-2}^{(3-\beta,\beta)}(x).\nonumber
\end{align*}	
We complete the proof by substituting the above equation in \eqref{TEMP32}.
\end{proof}

\begin{theorem}
\label{Thm_2}
The $(i,j)$-th element of the right fractional differentiation matrix 
${\mathlarger{\mathbf{\mathfrak D}}}_-^{2-\beta}$ 
is obtained explicitly as 
\begin{equation}
d_{ij}^-:= {}_{x}\hspace{-1pt}\mathcal 
D_{\text{\tiny 1} }^{2-\beta}\mathring{\ell}_j(\hat{\xi}_i)
=\sum_{k=1}^n\lambda_{jk} \, \zeta_k(1-\hat{\xi}_i),
\end{equation}
where
\begin{align*}
&\zeta_{k}(x):=
\tfrac{\Gamma(k+1)}{\Gamma(k-1+\beta)}x^{\beta-1}\hat P_{k-1}^{(3-\beta,\beta-1)}(x)\\
&-\tfrac{\Gamma(k+2)}{\Gamma(k+\beta)}\tfrac{k+2}{2k+1}x^{\beta}\hat P_{k-1}^{(3-\beta,\beta)}\hspace{-2pt}(x)\hspace{-2pt}-\hspace{-2pt}\tfrac{\Gamma(k+1)}{\Gamma(k-1+\beta)}
\tfrac{k}{2k+1}x^{\beta}\hat P_{k-2}^{(3-\beta,\beta)}\hspace{-2pt}(x).
\end{align*}
\end{theorem}

\begin{proof}
The proof is fairly similar to the proof of Theorem~\ref{Thm_1}. 
At first, by using property \eqref{Jac_Prop_2} for $a=2$ and $b=1$, we get
\begin{equation*}
\begin{split}
x&(1-x)P_{k-1}^{(1,1)}(2x-1)\\
&=(1-x)P_{k-1}^{(1,1)}(2x-1)-(1-x)^2P_{k-1}^{(1,1)}(2x-1)\\
&=(1-x)P_{k-1}^{(1,1)}(2x-1)-\tfrac{k+2}{2k+1}(1-x)^2P_{k-1}^{(2,1)}(2x-1)\\
&\qquad -\tfrac{k}{2k+1}(1-x)^2P_{k-2}^{(2,1)}(2x-1).
\end{split}
\end{equation*}
Using the above equation and Theorem~\ref{Zayer13_2}, we conclude that
\begin{align*}
&{}_{x}\hspace{-1pt}\mathcal D_{\text{\tiny 1} }^{2-\beta}
\left[x(1-x)P_{k-1}^{(1,1)}(2x-1)\right]\\
&\quad=\tfrac{\Gamma(k+1)}{\Gamma(k-1+\beta)}(1-x)^{\beta-1}P_{k-1}^{(\beta-1,3-\beta)}(2x-1)\\
&\qquad-\tfrac{\Gamma(k+2)}{\Gamma(k+\beta)}
\tfrac{k+2}{2k+1}(1-x)^{\beta}P_{k-1}^{(\beta, 3-\beta)}(2x-1)\\
&\qquad-\tfrac{\Gamma(k+1)}{\Gamma(k-1+\beta)}
\tfrac{k}{2k+1}(1-x)^{\beta}P_{k-2}^{(\beta,3-\beta)}(2x-1)\\
&\quad=\zeta_k(1-x).
\end{align*}	
The proof is immediately concluded by collocating the above equation 
at points $\hat{\xi}_i$, $i=1,\ldots,n$.
\end{proof}


\section{Step II: The Legendre--Gauss--Radau pseudospectral method for time discretization}
\label{sec:4}

In the second step of our method, the optimal control problem \eqref{main-OOCP} 
is solved numerically. This problem can be solved by any well-developed method, 
such as direct and indirect shooting methods. However, in this paper, we use the 
Legendre--Gauss--Radau pseudospectral method \citep{Garg10A}. 
Let $m$ be a positive integer number and $\hat \tau_j$, $j=1,\ldots,m$, 
be the Legendre--Gauss--Radau points in the interval $[0,T]$, i.e.,
$
\hat \tau_j:=\tfrac T 2\left(\tau^{(0,0)}_j+1\right).
$
Note that $\hat \tau_1=0$ and $\hat \tau_j<T$, $j=2,\ldots,m$. 
We define the extra node $\hat \tau_{m+1}=T$. Now, to solve the 
optimal control \eqref{main-OOCP}, we approximate 
the state function $\mathbf y(t)$ as
\begin{equation}
\label{appyj}
\mathbf y(t)\simeq \sum_{j=1}^{m+1} \mathbf y_j \bar{\ell}_j(t),
\end{equation}
where $\bar{\ell}_j(t)$, $j=1,\ldots,m,m+1$, are the Lagrange polynomials 
based on $\left\{ \hat{\tau_j}  \right\}_{j=1}^{m+1}$ and
$\mathbf y_1,\dots,\mathbf y_{m+1}$ are unknown $n$-vectors.
From the Kronecker property, and using the initial condition \eqref{iniooocp}, 
the vector $\mathbf y_1$ can be obtained explicitly as
\begin{equation}
\label{inidis}
\mathbf y_1=\mathbf g.
\end{equation}
By substituting the approximation \eqref{appyj} 
in the dynamic equation \eqref{dynooocp}, we get
\begin{equation*}
\sum_{j=1}^{m+1} \mathbf y_j \frac {\text{d}}{\text{d}t}{\bar{\ell}}_j(t)
= \mathbf C(t)  \, {\mathlarger{\mathbf{\mathfrak D}}}_{\pm}^{2-\beta}\, 
\sum_{j=1}^{m+1} \mathbf y_j \bar{\ell}_j(t) + \mathbf f(t)+ \mathbf u(t).
\end{equation*}
Now, by collocating the above equation in $\hat{\tau}_i$ for $i=1,\ldots,m$, we get
\begin{equation*}
\sum_{j=1}^{m+1} \mathbf y_j \frac {\text{d}}{\text{d}t}{\bar{\ell}}_j(\hat{\tau}_i)
= \mathbf C(\hat{\tau}_i)  \, {\mathlarger{\mathbf{\mathfrak D}}}_{\pm}^{2-\beta}\, 
\sum_{j=1}^{m+1} \mathbf y_j \bar{\ell}_j(\hat{\tau}_i) + \mathbf f(\hat{\tau}_i)+\mathbf u_i,
\end{equation*}
where $\mathbf u_i:=\mathbf u(\hat{\tau}_i)$. It is worthwhile to note that 
$\hat{\tau}_{m+1}$ is used in approximating $\mathbf y(t)$, however, this point 
is not used as a collocation point. Using the Kronecker property, 
from the above equation, we get that
\begin{equation}
 \hspace{-3pt}\sum_{j=1}^{m+1} \hspace{-3pt}\mathbf y_j \bar{d}_{ij}
 \hspace{-2pt}= \hspace{-2pt}\mathbf C(\hat{\tau}_i)  \, {\mathlarger{\mathbf{\mathfrak D}}}_{\pm}^{2-\beta}
\mathbf y_i  \hspace{-2pt} +  \hspace{-2pt}\mathbf f(\hat{\tau}_i) \hspace{-2pt}+ \hspace{-2pt} \mathbf u_i,\  i=1,\dots,m,\label{jkl}
\end{equation}
where
$
\bar{d}_{ij}=\frac {\text{d}}{\text{d}t}{\bar{\ell}}_j(\hat{\tau}_i).
$
According to  \citet{BT2003}, for stable and accurate computation of 
$\bar{d}_{ij}$, $i=1,\ldots,m$, $j=1,\ldots,{m+1}$, 
we can use the following formula:
\begin{align*}
\bar{d}_{ij}=\begin{cases}
\frac {\lambda_j} {\lambda_i(\hat{\tau}_j-\hat{\tau}_i)},
& i\ne j,\\
\sum_{k=1, k\ne i}^m \bar{d}_{ik},
& i=j,
\end{cases}
\end{align*} 
where $\lambda_i=\prod_{k=1,k\ne i}^{m} (\hat{\tau}_i-\hat{\tau}_k)$.

Next, the path constraint \eqref{pathooocp} is collocated at the $m$ 
collocation points $\hat{\tau}_1,\dots,\hat{\tau}_m$ and, finally, 
the following inequality constraints are obtained:
\begin{equation}
\mathbf u_i\ge \mathbf u_{\min}(\hat{\tau}_i),
\qquad i=1,\dots,m.
\end{equation}
By using the Legendre--Gauss--Radau quadrature \eqref{JGLQ} with $a=b=0$, 
the performance index \eqref{objooocp} is approximated by
\begin{align*}
J_{n,m}=\tfrac{1}{2}&\sum_{j=1}^{m}\tfrac T 2{\varpi}^{\text{\tiny{(0,0)} }}_j\big[
\mathbf y_j-\mathbf z(\hat{\tau}_j)\big]^{\text{T}}\mathbf W 
\big[\mathbf y_j-\mathbf z(\hat{\tau}_j)\big] \\
&+\tfrac{1}{2}\sum_{j=1}^{m}\tfrac T 2{\varpi}^{\text{\tiny{(0,0)} }}_j
\mathbf u_j^\text{T} \mathbf W  \mathbf u_j,
\end{align*}
where ${\varpi}^{\text{\tiny{(0,0)} }}_j$, $j=1,\ldots,m$, 
are the Legendre--Gauss--Radau weights.

In summary, by applying the Legendre--Gauss--Radau method, 
the optimal control problem \eqref{main-OOCP} 
is transcribed into the following optimization problem:
\begin{subequations}
\label{main-QP}
\begin{align}
&\min \ \  J_{n,m}\nonumber\\
&\quad\ \ \ =\tfrac{1}{2}\sum_{j=1}^{m}\tfrac T 2{\varpi}^{\text{\tiny{(0,0)} }}_j\big[
\mathbf y_j-\mathbf z(\hat{\tau}_j)\big]^{\text{T}}\mathbf W 
\big[\mathbf y_j-\mathbf z(\hat{\tau}_j)\big]\nonumber\\
&\qquad\qquad\qquad+\tfrac{1}{2}\sum_{j=1}^{m}\tfrac T 2{\varpi}^{\text{\tiny{(0,0)} }}_j
\mathbf u_j^\text{T} \mathbf W  \mathbf u_j, \label{objQP}\\
&\ \ \text{s.t.}\ \ \sum_{j=1}^{m+1} \mathbf y_j \bar{d}_{ij}
= \mathbf C(\hat{\tau}_i)  \, {\mathlarger{\mathbf{\mathfrak D}}}_{\pm}^{2-\beta}\, 
\mathbf y_i  + \mathbf f(\hat{\tau}_i)+ \mathbf u_i,\\
&\qquad\qquad\quad i=1,\dots,m,\nonumber\\
&\quad\quad\  \mathbf y_1=\mathbf g,\label{iniQP}\\
&\quad\quad\ \mathbf 
u_i\ge \mathbf u_{\min}(\hat{\tau}_i),\qquad i=1,\dots,m,\label{pathQP}
\end{align}
\end{subequations}
where the decision variables are the $n$-vectors $\mathbf u_i$, $i=1,\ldots,m$, 
and $\mathbf y_j$, $i=1,\ldots,m+1$. We note that the objective function 
is quadratic and the constraints are affine, that is, problem \eqref{main-QP}
is a quadratic programming problem (QP).


\subsection*{Converting the QP into standard form}

As we have seen, with our method the numerical solution of problem 
\eqref{main-prob} is reduced to solving the QP \eqref{main-QP}. However, 
this QP is not in standard form. In the following, we derive the standard form of 
\eqref{main-QP}, which helps us to easily analyze and utilize a solver on it.

At first, we consider the objective function \eqref{objQP} and expand it as
\begin{multline}
\label{jklmnb}
J_{n,m}
=\tfrac{1}{2}\hspace{-3pt}\sum_{j=1}^{m}\hspace{-3pt}
\tfrac T 2{\varpi}^{\text{\tiny{(0,0)} }}_j \mathbf y_j^{\text{T}}\mathbf W 
\mathbf y_j +\tfrac{1}{2}\hspace{-3pt}\sum_{j=1}^{m}\hspace{-3pt}\tfrac T 2{\varpi}^{\text{\tiny{(0,0)} }}_j
\mathbf u_j^\text{T} \mathbf W  \mathbf u_j \\
\hspace*{-12pt}-\sum_{j=1}^{m}\hspace{-3pt}\tfrac T 2{\varpi}^{\text{\tiny{(0,0)} }}_j 
\mathbf z^{\text{T}}(\hat{\tau}_j)\mathbf W \mathbf y_j 
+\tfrac 1 2\hspace{-3pt}\sum_{j=1}^{m}\hspace{-3pt}\tfrac T 2{\varpi}^{\text{\tiny{(0,0)} }}_j
\mathbf z^{\text{T}}(\hat{\tau}_j)\mathbf W \mathbf z(\hat{\tau}_j).
\end{multline}
If we define the matrices  $\mathbf Y \in \mathbb R^{n\times(m+1)}$, 
$\mathbf U \in \mathbb R^{n\times m}$, 
$\mathbf Z \in \mathbb R^{n\times m}$, 
the $n$-vector $\boldsymbol{\omega}$ 
and the $m$-vector $\boldsymbol{\varpi}$ as
\begin{align*}
&\mathbf Y
:=\begin{bmatrix}
\mathbf y_{1}&\dots&\mathbf y_{m}&\mathbf y_{m+1}
\end{bmatrix},\ \ 
\mathbf U:=\begin{bmatrix}
\mathbf u_1&\dots&\mathbf u_m
\end{bmatrix},\\
&\mathbf Z:= 
\begin{bmatrix}
\frac{z(\hat{\xi}_i,\hat{\tau}_j)}{{\hat{\xi}_i\left(1-\hat{\xi}_i\right)}}
\end{bmatrix}_{\substack{i=1,\dots,n\\j=1,\dots,m}},\\
&\boldsymbol{\omega}:=\tfrac 1 2 \hspace{-3pt}
\begin{bmatrix}
\omega^{\text{\tiny(1,1)}}_1
&\dots& \omega^{\text{\tiny(1,1)}}_n
\end{bmatrix}^{\text{T}},\,
\boldsymbol{{\varpi}}:= \hspace{-3pt}\tfrac T 2
\begin{bmatrix}
{\varpi}^{\text{\tiny{(0,0)} }}_1
&\dots&{\varpi}^{\text{\tiny{(0,0)} }}_n
\end{bmatrix}^{\text{T}}
\end{align*}
then, the objective function \eqref{jklmnb} 
can be rewritten as 
\begin{align*}
&J_{n,m}=\tfrac{1}{2} \textbf{vec}\left(\mathbf Y\right)^{\text{T}} 
\mathbf{\bar S}\; \textbf{vec}\left(\mathbf  Y\right)
+\tfrac{1}{2}\textbf{vec}\left(\mathbf U\right)^{\text{T}} 
\mathbf S \;\textbf{vec}\left(\mathbf U\right)\\
&\qquad\quad-\textbf{vec}\left(\mathbf Z\right)^{\text{T}} 
\mathbf {\bar S} \;\textbf{vec}\left(\mathbf Y\right)
+\tfrac{1}{2}\textbf{vec}\left(\mathbf Z\right)^{\text{T}} 
\mathbf S \;\textbf{vec}\left(\mathbf Z\right),
\end{align*}
where
\begin{equation*}
\mathbf S:=\text{diag}\left(\boldsymbol{{\varpi}}
\otimes \boldsymbol{\omega} \right),
\qquad
\mathbf{\bar S}:=\left[
\begin{array}{c|c}
\ 	\mathbf S\ \ 
& \mathbf 0_{nm\times n}\\
\hline
\mathbf 0_{n\times nm} 
& \mathbf 0_{n\times n} \\
\end{array}\right].
\end{equation*}
If we collect all decision variables in a vector $\mathbf v$ as
\begin{equation}
\label{vvar}
\mathbf v:=
\begin{bmatrix}
\textbf{vec}\left(\mathbf Y\right)\\
\textbf{vec}\left(\mathbf U\right)
\end{bmatrix}
=
\left[
\begin{array}{c}
\mathbf y_1\\ \vdots\\ \mathbf y_{m+1} \\ \hline 
\mathbf u_1\\ \vdots \\ \mathbf u_m
\end{array}
\right],
\end{equation}
then the performance index \eqref{objQP} can be expressed 
as the following function of $\mathbf v$:
\begin{equation*}
J_{n,m}(\mathbf v)=\tfrac 1 2 \mathbf v^{\text{T}} 
\mathbf H \mathbf v +\mathbf c^{\text{T}}\mathbf v +c_0,
\end{equation*}
where
\begin{align*}
&\mathbf H:=\left[
\begin{array}{c|c|c}
\ 	\text{diag}\left(\boldsymbol{{\varpi}}\otimes \boldsymbol{\omega} \right)\ \ 
& \mathbf 0_{nm\times n}& \mathbf 0_{nm\times nm}\\ \hline
\mathbf 0_{n\times nm} & \mathbf 0_{n\times n} & \mathbf 0_{n\times nm}\\ \hline
\mathbf 0_{nm\times nm} & \mathbf 0_{nm\times n}
&\  \text{diag}\left(\boldsymbol{{\varpi}}\otimes \boldsymbol{\omega} \right)\ 
\end{array}\right],\\
&\mathbf c:=-\left[\begin{array}{c}
\mathbf {\bar S}^{\text{T}} \textbf{vec}\left(\mathbf Z\right)   \\ \hline 
\mathbf 0_{n\times 1}\\ \hline \mathbf 0_{nm\times 1}
\end{array}
\right],\ \ c_0:=\tfrac{1}{2}\textbf{vec}\left(\mathbf Z\right)^{\text{T}} 
\mathbf S \;\textbf{vec}\left(\mathbf Z\right).
\end{align*}
In order to reformulate the constraints of the QP \eqref{main-QP} 
as standard linear constraints, we define $\mathbf D\in \mathbb R^{m\times(m+1)}$, 
$\mathbf{\bar C} \in \mathbb R^{m\times(m)}$, 
and $\mathbf{F} \in \mathbb R^{m\times(m)}$ as follows:
\begin{align*}
&\mathbf D:=
\Big[
\bar{d}_{jk}
\Big]_{\substack{j=1,\dots,m\ \ \\k=1,\dots,m+1}},\\
&\mathbf{\bar C} :=
\Big[
c(\hat{\xi}_i,\hat{\tau}_j)
\Big]_{\substack{i=1,\dots,n\\j=1,\dots,m}},
\quad
\mathbf{F} :=
\Big[
f(\hat{\xi}_i,\hat{\tau}_j)
\Big]_{\substack{i=1,\dots,n\\j=1,\dots,m}}.
\end{align*}
Considering the above notations, we can write all $m$ equations 
in \eqref{jkl} in the following matrix equation:
\begin{equation}
\label{meq}
\mathbf Y \mathbf D^{\text{T}} =\mathbf{\bar C} 
\odot\left( {\mathlarger{\mathbf{\mathfrak D}}}_{\pm}^{2-\beta}  
\,[\mathbf Y]_{1:m}\right)+ \mathbf F+\mathbf U,
\end{equation}
where $\odot$ refers to element-wise or Hadamard product and 
$[\mathbf Y]_{1:m} $ denotes the matrix obtained by removing 
the last column of $\mathbf Y$. Let $\mathbf I_n$ denote the 
identity matrix of dimension $n$ and $\mathbf{\bar I}_m$ 
denote the matrix obtained by removing the last column 
of the identity matrix $\mathbf{I}_{m+1}$. By noting that
$\mathbf Y \mathbf D^{\text{T}}=\mathbf I_n\mathbf Y \mathbf D^{\text{T}}$
and $[\mathbf Y]_{1:m} =\mathbf Y\,\mathbf{\bar I}_m$, 
the matrix equation \eqref{meq} can be written as
\begin{equation*}
\mathbf I_n \mathbf Y \mathbf D^{\text{T}} 
= \mathbf{\bar C} \odot\left( {\mathlarger{\mathbf{\mathfrak D}}}_{\pm}^{2-\beta}  
\,\mathbf Y \,\mathbf{\bar I}_m \right)
+\mathbf F+\mathbf U.
\end{equation*} 
Now, by using Theorem~\ref{KRON}, the above matrix equation 
is converted to the following linear system of equations:
\begin{align*}
\left(\mathbf D \otimes \mathbf I_n\right)\textbf{vec}\left(\mathbf Y \right)
= \mathbf I_{\mathbf C} &\left(\mathbf{\bar I}_m^{\text{T}} 
\otimes {\mathlarger{\mathbf{\mathfrak D}}}_{\pm}^{2-\beta} \right)
\textbf{vec}\left(\mathbf Y \right)\\
&\qquad+\textbf{vec}\left(\mathbf U\right)+
\textbf{vec}\left(\mathbf F\right),
\end{align*}
where $\mathbf I_{\mathbf C}:=\textbf{diag}\left(\textbf{vec}\left(\mathbf C\right)\right)$.
By utilizing \eqref{vvar}, we express the above system of equations as
\begin{equation*}
\begin{bmatrix}
\left(\mathbf D \otimes \mathbf I_n\right)-\mathbf I_{\mathbf C} 
\left(\mathbf{\bar I}_m^{\text{T}} \otimes {\mathlarger{
\mathbf{\mathfrak D}}}_{\pm}^{2-\beta} \right) 
&\Big|& \mathbf I_{nm}
\end{bmatrix} 
\mathbf v= \textbf{vec}\left(\mathbf F\right).
\end{equation*}
The constraints \eqref{iniQP} and \eqref{pathQP} are expressed as
\begin{align*}
&\begin{bmatrix}
\mathbf I_n&\big  |\mathbf 0_{n\times n(m-1)}
&\big |&\mathbf 0_{n\times nm}
\end{bmatrix}
\mathbf v = \mathbf g,\\
&\begin{bmatrix}
\mathbf 0_{nm\times nm}&\big |&\mathbf I_{nm}
\end{bmatrix}
\mathbf v 
\ge \textbf{vec}\left(\mathbf{U}_{\min}\right),
\end{align*}
where  
\begin{equation*}
\mathbf{U}_{\min}
:=
\Big[
u_{\min}(\hat{\xi}_i,\hat{\tau}_j)
\Big]_{\substack{i=1,\dots,n\\j=1,\dots,m}}.
\end{equation*}
In summary, the standard form of the QP \eqref{main-QP} 
is obtained as
\begin{subequations}
\label{main-sQP}
\begin{align}[left =  \empheqlbrace\,]
&\min \ \ J_{n,m}=\tfrac 1 2 \mathbf v^{\text{T}} 
\mathbf H \mathbf v +\mathbf c^{\text{T}}\mathbf v +c_0, \label{objsQP}\\
&\qquad\ \ \text{s.t.}\qquad\mathbf A \mathbf v=\mathbf b,\label{ceqsQP}\\
&\qquad\qquad\qquad\mathbf B\mathbf v\le \mathbf h,\label{pathsQP}
\end{align}
\end{subequations}
where
\begin{align*}
&\mathbf A:=
\left[	
\begin{array}{l|l}
\mathbf I_n\  \big  |\qquad\qquad \mathbf 0_{n\times n(m-1)}\ 
& \quad \mathbf 0_{n\times nm}\\ \hline
\left(\mathbf D \otimes \mathbf I_n\right)-\mathbf I_{\mathbf C} 
\left(\mathbf{\bar I}_m^{\text{T}} \otimes 
{\mathlarger{\mathbf{\mathfrak D}}}_{\pm}^{2-\beta} \right) 
\ \ & \quad \mathbf I_{nm}
\end{array}
\right],\\
& \mathbf b:=\left[\begin{array}{c} \mathbf g\\ \hline
\textbf{vec}\left(\mathbf F\right)
\end{array}
\right],\ 
\mathbf B:=-\begin{bmatrix}
\mathbf 0_{n(m+1)\times n(m+1)}&\big |&\mathbf I_{nm}
\end{bmatrix},\\
&\mathbf h:=-\textbf{vec}\left(\mathbf U_{\min}\right).
\end{align*}
Note that the dimension of $\mathbf A$ is $n(m+1)\times n(2m+1)$, 
where the number of non-zero elements is $nm(n+m+1)+n$.
We conclude that the percentage of non-zero elements decreases 
dramatically when the size of the matrix $\mathbf A$ is increased.
For instance, the percentage of non-zero elements of $A$ obtained 
with $n=m=10$ and $n=m=100$ is $9.1\%$ and $1\%$, respectively. 
Thus, matrix $\mathbf A$ is a sparse matrix. The sparsity pattern 
of the matrix $\mathbf A$, which is obtained with $n=m=10$, 
is plotted in Figure~\ref{fig:fig1}.
\begin{figure*}
\centering
\includegraphics[width=\linewidth]{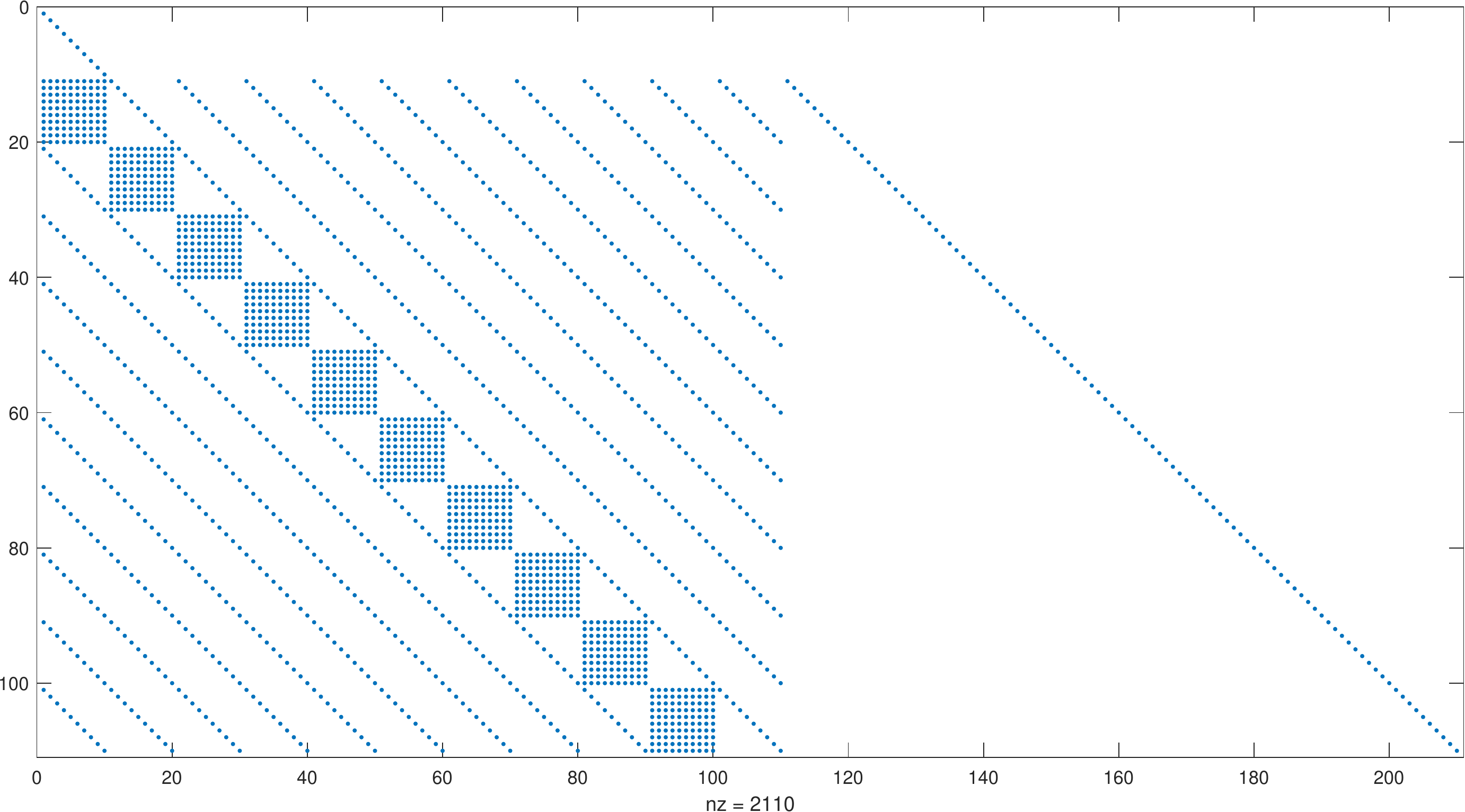}
\caption{The sparsity pattern of the matrix $\mathbf A$ with $n=m=10$.}
\label{fig:fig1}
\end{figure*}

The sparsity of the QP \eqref{main-sQP} is a suitable property 
from the optimization point of view. Another important 
and desirable property in optimization is convexity.
We emphasize that sparsity, and especially convexity, help to solve 
the obtained problem efficiently by well-developed algorithms, 
such as interior-point methods. In the next theorem, we show that 
the obtained problem \eqref{main-sQP} is convex. 

\begin{theorem}
The optimization problem \eqref{main-sQP} is a convex problem.
\end{theorem}

\begin{proof}
Since the quadrature weights ${\omega}^{\text{\tiny{(1,1)} }}_i$, $i=1,\ldots,n$, 
and ${\varpi}^{\text{\tiny{(0,0)} }}_j$, $j=1,\ldots,m$, are positive, 
it follows that $\mathbf H$ is a diagonal matrix with non-negative 
diagonal elements \citep{Gautschi}. Consequently, $\mathbf H$ is a positive 
semi-definite matrix. As a result, the objective function is a convex function. 
Moreover, the equality and inequality constraints of problem \eqref{main-sQP} are affine. 
As a result, problem \eqref{main-sQP} is a convex quadratic program.
\end{proof}


\section{Numerical experiments}
\label{sec:5}

This section is devoted to illustrate the presented pseudospectral 
method using numerical experiments. We have implemented our method  
using \textsc{Matlab} on a 3.5 GHz Core i7 personal computer 
with 8 GB of RAM. Moreover, for solving the QP \eqref{main-sQP}, 
the solver \textsc{Ipopt} \citep{Waechter2006} was used. 
In \textsc{Ipopt}, we can adjust the accuracy of the solution 
through the input parameter \texttt{tolrfun}. In our numerical 
experiments, we set \texttt{tolrfun}=$10^{-12}$.

We consider five different examples. 
The first example was treated in \citet{Ning} 
and has an exact solution. We use this example 
to assess the accuracy of our method. 
The second example is considered in  
\citet{doi:10.1177/1077546317705557} and has 
a high variation in its solution. With this example 
we check the efficiency of our method.
The third example is a new problem with state constraints 
and shows that, with little changes in the presented method, 
we can easily solve hard problems.
With the last two examples, we investigate the effect 
of parameter $\beta$ on the behavior of the solution 
and the performance of our method.


\subsection*{Example~1}

In this example, which is taken from  \citet{Ning}, 
the problem  \eqref{main-prob} with the following data is considered:
\begin{align*}
&\beta=0.2,\ r=0.8,\ T=1,\ g(x)=0,\ u_{\min}(x,t)=1,\\
& c(x,t)=\tfrac{1+xt}{100},\, f(x,t)=-\max\hspace{-2pt}\left\{\tfrac{100x^2(1-x)^2 \sin(T-t)}{1+xt},1\right\},\\
&z(x,t)=\tfrac{100x^3(1-x)^2\sin(T-t)}{(1+xt)^2}+\tfrac{100x^2(1-x)^2\cos(T-t)}{1+xt}\\
&-\tfrac{2\sin(T-t)}{5\Gamma(1.2)}\big[(x^{0.2}+4(1-x)^{0.2}-5(x^{1.2}+4(1-x)^{1.2})\\ &\qquad\qquad+\tfrac{50}{11}(x^{2.2}+4(1-x)^{2.2})\big].
\end{align*}
The exact solution of this problem is
\begin{align*} 
&y_{ex}(x,t)=0, 
&u_{ex}(x,t)=\max\left\{\tfrac{100x^2(1-x)^2 \sin(T-t)}{1+xt},1\right\}.
\end{align*}
By applying the presented method with $m=n=50$, the obtained control 
$u$ and state $y$ are plotted in Figure~\ref{fig:fig2}. 
\begin{figure*}[ht]
\centering
\includegraphics[width=1\textwidth]{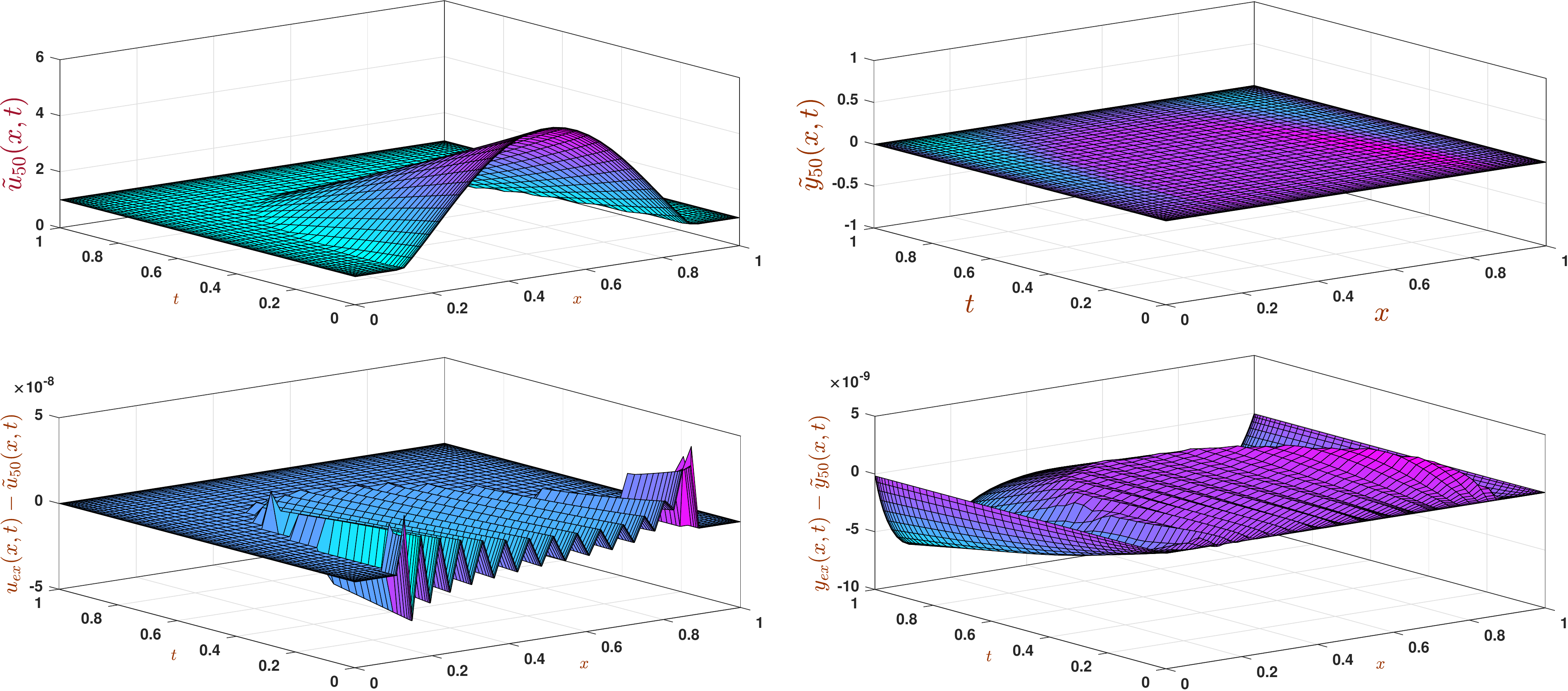}
\caption{Control, state and error functions obtained 
by the proposed method with $n=m=50$ (Example~1).}
\label{fig:fig2}
\end{figure*}
Moreover, the errors of the obtained control and state functions are plotted in this figure too. 
In Table~\ref{tab:tab1}, the consumed CPU time, error of the obtained value of performance 
index and  error of the obtained control and state functions, 
for various values of $n$ and $m$, are reported. 
In this table, $E_{n,m}(J)$ is the absolute error of the value 
of the performance index, i.e., 
$E_{n,m}(J):=\mid J_{\text{ex}}-J_{n,m} \mid$
and $E_{n,m}^2(u)$ and $E_{n,m}^\infty(u)$ are, respectively,
the $2$-norm and infinity-norm of the error for the obtained 
state function, which are defined as
\begin{align*}
&E_{n,m}^2(u):=\left[ \sum_{i=1}^n\sum_{j=1}^m 
\left(u_{\text{ex}}(\hat \xi_i,\hat \tau_j)
-u_{i,j} \right)^2 \right]^{\frac 1 2},\\
&E_{n,m}^\infty(u):=\max_{\substack{i=1,\dots,n\\
j=1,\dots,m}} \mid u_{\text{ex}}(\hat \xi_i,\hat \tau_j)-u_{i,j} \mid.
\end{align*}
\begin{table*}[ht]
\centering
\caption{\label{tab:tab1}
Consumed CPU time and obtained norms of errors for various values of 
$n$ and $m$ together with the results of \citet{Ning} (Example~1).}
\begin{tabular}{lllllllllllll}
\hline
\multicolumn{8}{c}{The Presented Method} &  & \multicolumn{4}{c}{Method of \cite{Ning} (PCG)} \\ 
\cline{1-8}\cline{10-13}
$n$ & $m$ & CPU & $E_{n,m}(J)$ & $E^2_{n,m}(u)$ & $E^\infty_{n,m}(u)$ 
& $E^2_{n,m}(y)$ & $E^\infty_{n,m}(y)$ &  & $N=M$ & CPU 
& $\|u_{\text{ex}}-u_h\|_{l^2}$ & $\|y_{\text{ex}}-y_h\|_{l^2}$ \\ 
\cline{1-8}\cline{10-13}
10 & 10 & 0.7 & 1.24e-3 & 3.67e-5 & 1.38e-4  & 7.01e-6 & 2.18e-5 &  & 32 & 0.25 & 1.8638e-4 & 9.5396e-5 \\ 
20 & 20 & 0.6 & 5.71e-4 & 9.36e-7  & 5.23e-6 & 7.65e-8  & 3.27e-7 &  & 64 & 0.92 & 4.7116e-5 & 2.4422e-5 \\ 
30 & 30 & 1.0 & 1.48e-5  & 9.24e-8 & 3.91e-7 & 5.56e-9 & 2.50e-8 &  & 128 & 3.77 & 1.1825e-5 & 6.1531e-6 \\ 
40 & 40 & 2.5 & 2.24e-6  & 1.98e-8 & 8.79e-8 & 2.33e-9 & 6.82e-9 &  & 256 & 15 & 2.9647e-6 & 1.5343e-6 \\ 
50 & 50 & 7.1 & 6.07e-7 & 6.28e-9 & 4.84e-8 & 2.30e-9 &  6.62e-9 &  & 512 & 62 & 7.5488e-7 & 3.7709e-7 \\ 
60 & 60 & 14.9 & 1.00e-7 & 2.50e-9 & 2.34e-8 & 2.37e-9 & 6.51e-9 &  & 1024 & 257 & 2.2979e-7 & 9.3049e-8 \\ 
\hline
\end{tabular}
\end{table*}
From Table~\ref{tab:tab1}, we can see, 
with small numbers of $n$ and $m$, that an accurate solution 
with low CPU time is obtained.
Moreover, it is  seen that with $n=m=50$, an accurate  solution is obtained 
in just 7.1 seconds. Note that by the indirect method of \citet{Ning},
more than 4 minutes are needed for obtaining a solution with such accuracy.


\subsection*{Example 2} 

This example, with the following data, 
is treated in  \citet{doi:10.1177/1077546317705557}:
\begin{align*}
&\beta=0.5,\ r=0.25,\ T=30,\ y(x,0)=1,\\  
&c(x,t)=\tfrac{1+x(1-x)t}{10},\  f(x,t)=1, \ z(x,t)=1+\tfrac{x}{1+t},\\
&u_{\min}(x,t)=\max\left\{xe^{-2(x-0.5)},\sin(2x(1-x)t^{0.6})\right\}.
\end{align*}
The resulted control and state functions by the presented method, with 
$n=m=100$, are plotted in Figure~\ref{fig:fig3}. 
\begin{figure*}[ht]
\centering
\includegraphics[width=1\textwidth]{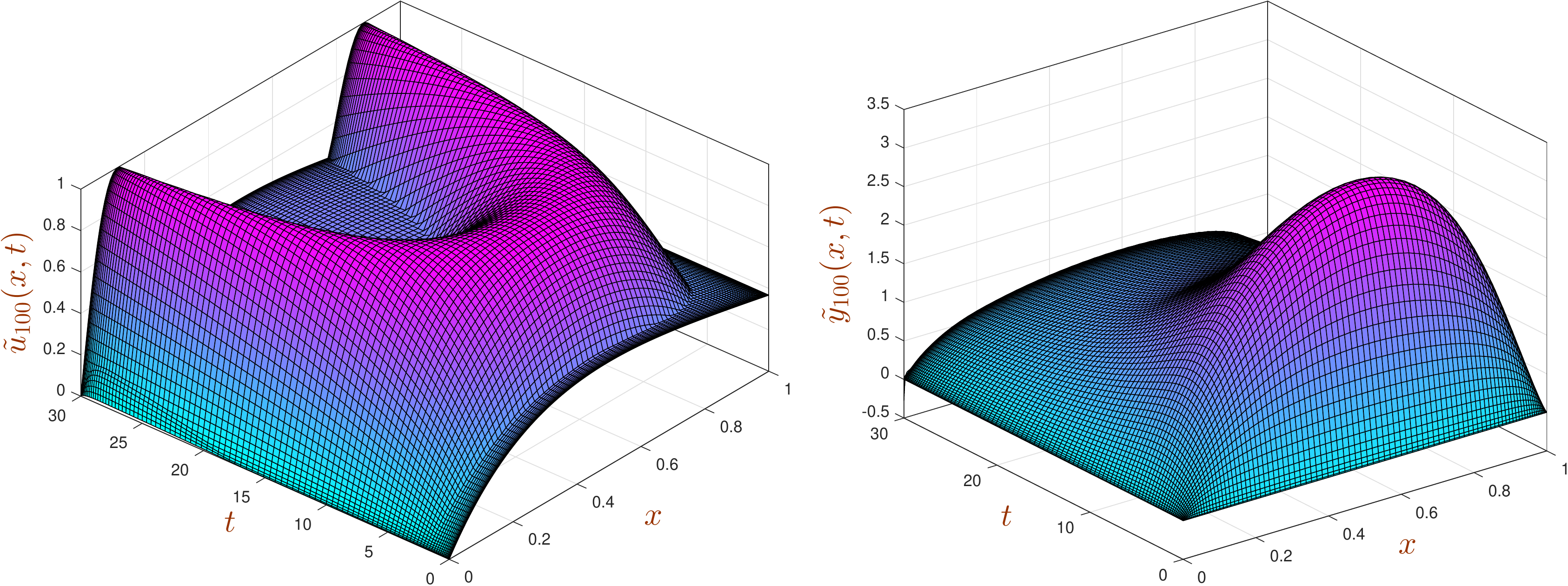}
\caption{Control and state functions obtained by the proposed method with $n=m=100$ (Example~2).}
\label{fig:fig3}
\end{figure*}
The obtained control and state functions are in good agreement 
with the results of \citet{doi:10.1177/1077546317705557}.
To report the efficiency and precision of the method, 
we give in Table~\ref{tab2} the CPU time and 
the obtained values of the performance index. 

\begin{table*}[ht]
\centering
\caption{\label{tab2}
Consumed CPU time and obtained values of the performance index 
for various values of $n$ and $m$ (Example~2).}
\begin{tabular}{lllllll} \hline
\multicolumn{1}{c}{$n$}& \multicolumn{1}{c}{} & \multicolumn{1}{c}{$m$} 
& \multicolumn{1}{c}{} & \multicolumn{1}{c}{CPU Time} 
& \multicolumn{1}{c}{} & \multicolumn{1}{c}{$J_{n,m}$} \\ \hline
20 &  & 20 &  & 0.6 s  &  & 17.3176450\\ 
40 &  & 40 &  & 3.2 s  &  & 17.3055445\\ 
60 &  & 60 &  & 22.7 s &  & 17.3028343\\ 
80 &  & 80 &  & 100.6 s&  & 17.3028684 \\ 
100 && 100 &  & 311.6 s&  & 17.3028411 \\ \hline
\end{tabular}
\end{table*}


\subsection*{Example~3}

Now we consider the problem of Example~1 subject to the extra state constraint 
$y(x,t)\ge y_{\min}(x,t)$, where
\begin{equation*}
y_{\min}(x,t):=\sqrt{\max\left\{ 0,0.1-(x-0.5)^2-(t-0.5)^2 \right\}}.
\end{equation*}  
The numerical solution of this problem with indirect methods is troublesome
because the derivation of necessary optimality conditions for such optimal control
problems with state constraints is difficult. Moreover, the numerical solution 
of the resulted necessary optimality conditions is complicated. 
However, extending our direct method to solve such problems is straightforward.  
We just need to add the following bound constraints to the QP \eqref{main-QP}:
\begin{equation*}
\mathbf y_j\ge \mathbf y_{\min}(\hat{\tau}_j), 
\quad j=1,\dots,m+1,
\end{equation*}
where $\mathbf y_{\min}(t):=\big[y_{\min}(\hat{\xi}_1,t),\dots,y_{\min}(\hat{\xi}_n,t)\big]^{\text{T}}$.
Imposing the above constraints is simply done by changing the definition of 
$\mathbf B$ and $\mathbf h$ in the QP \eqref{main-sQP} as follows:
\begin{equation*}
\mathbf B:=-\begin{bmatrix}
\mathbf I_{n(m+1)}&\big |&\mathbf I_{nm}
\end{bmatrix},\qquad
\mathbf h:=-\begin{bmatrix}
\textbf{vec}\left(\mathbf Y_{\min}\right)\\ \hline
\textbf{vec}\left(\mathbf U_{\min}\right)
\end{bmatrix},
\end{equation*}
where 
\begin{equation*}
\mathbf Y_{\min}:=
\Big[
y_{\min}(\hat{\xi}_i,\hat{\tau}_j)
\Big]_{\substack{i=1,\dots,n\ \ \\j=1,\dots,m+1}}.
\end{equation*}
By applying our direct method with $n=m=100$, the obtained control and state 
functions are plotted in Figure~\ref{fig:fig4}. Moreover, function 
$y_{\min}(x,t)$ is plotted beside the state function $y_{100}(x,t)$. 
It is clear that in this example, contrary to Example~1, 
the state function is not zero and lies above function $y_{\min}(x,t)$.
\begin{figure*}
\centering
\includegraphics[width=\textwidth]{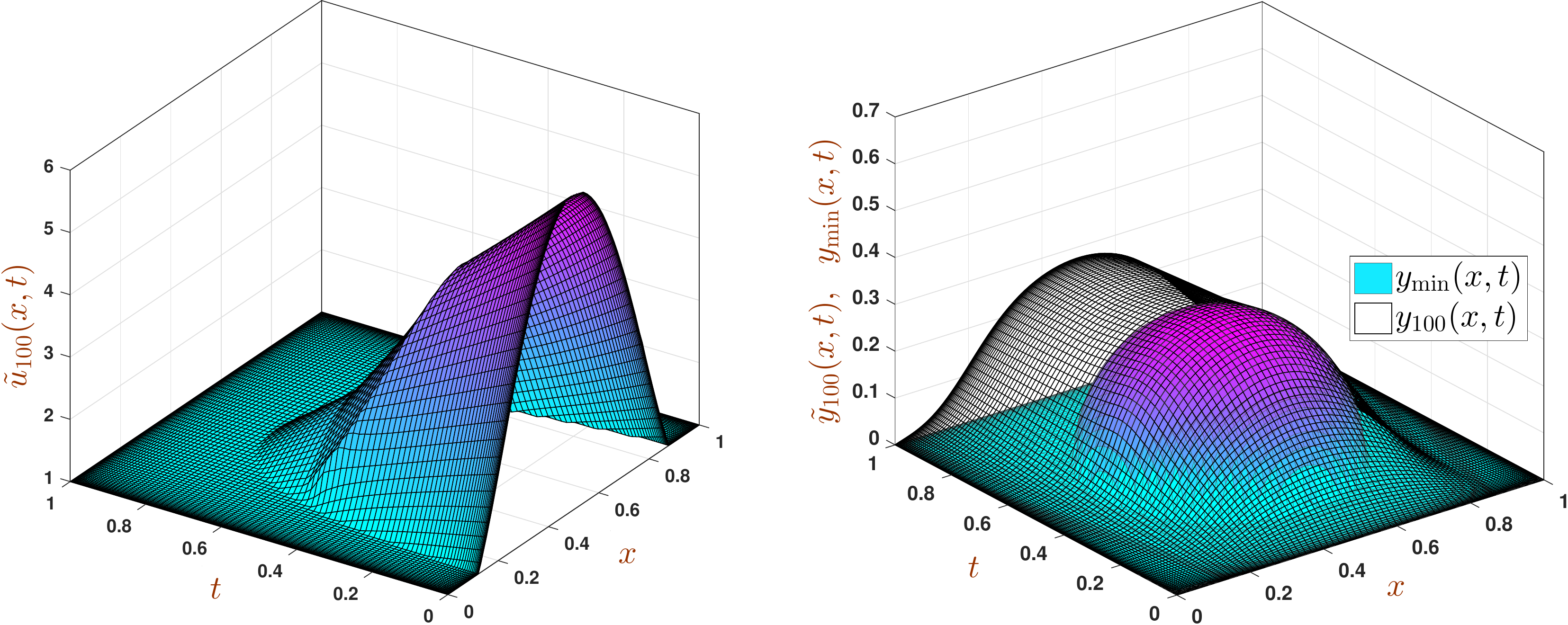}
\caption{Control and state functions obtained with $n=m=100$
for problem of Example~1 subject to a state constraint 
$y(x,t)\ge y_{\min}(x,t)$ (Example~3).}
\label{fig:fig4}
\end{figure*}


\subsection*{Example~4}

Consider now problem \eqref{main-prob} with
\begin{align*} 
&r=0.5,\ T=3,\ y(x,0)=\sin(\pi x),\ c(x,t)=1,\\  
&f(x,t)=0, \ z(x,t)=0.5,\ u_{\min}(x,t)=0.
\end{align*}
In Figure~\ref{fig:fig5}, the obtained solutions with $m=n=60$ 
on this problem, for $\beta=0.1,\,0.5,\,0.9$, are plotted. 
\begin{figure*}
\centering
\includegraphics[width=\textwidth]{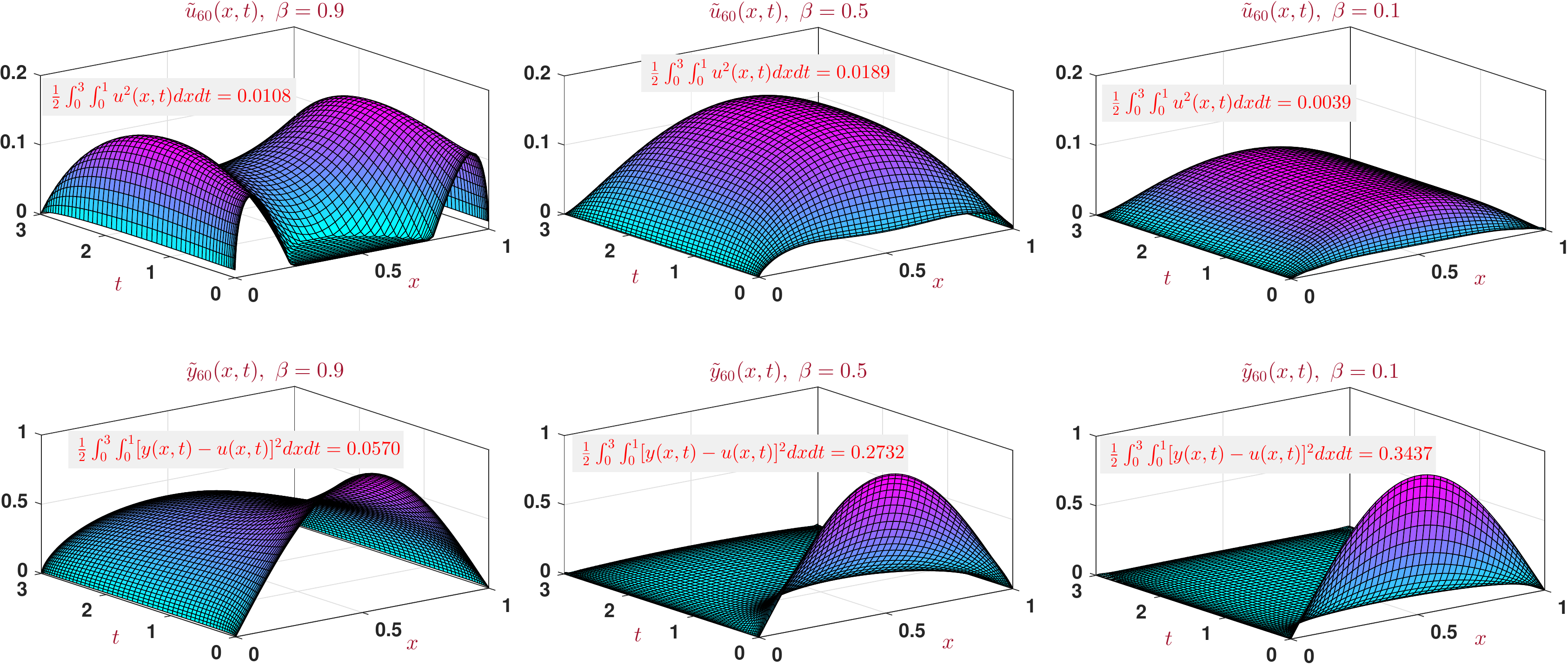}
\caption{Control and state functions obtained by the presented method 
with $n=m=60$, for $\beta=0.1,\,0.5,\,0.9$ (Example~4).}
\label{fig:fig5}
\end{figure*}
Moreover, the values of the first and second terms in the performance index 
\eqref{main-prob-obj} are reported as well. It is seen, 
for smaller values of $\beta$, that the obtained state $\tilde y_{60}$ 
is not close to the target function $z(x,t)=\tfrac 1 2$. Indeed, 
when the value of $\beta$ is small, then a large magnitude control function 
is needed to generate a state function close to $z(x,t)$. In other words, 
much energy is needed to force $y(x,t)$ to become close to $z(x,t)$. 
As a result, in these cases, the optimal control is close to its lower bound 
and the the corresponding state function is not close to the target function. 


\subsection*{Example~5}

As a final example, let
\begin{align*}
&r=0.5,\ T=3,\ y(x,0)=x^4(1-x)^4, u_{\min}(x,t)=1,\\  
&c(x,t)=1,\ z(x,t)=e^tx^4(1-x)^4,\\
& f(x,t)=e^t\left[x^4(1-x)^4-12 \left( \bar f(x)+\bar f(1-x)\right)\right],
\end{align*}
where
\begin{equation*}
\bar f(x)=\hspace{-2pt}\tfrac {1680 {x}^{6+\beta}} {\Gamma \left( 7+\beta \right) }\hspace{-1pt}-\hspace{-1pt}
\tfrac {840 {x}^{5+\beta}} {\Gamma \left( 6+\beta \right) }\hspace{-1pt}+\hspace{-1pt}
\tfrac{180 {x}^{4+\beta}}{\Gamma \left( 5+\beta \right) }\hspace{-1pt}-\hspace{-1pt}
\tfrac {20	{x}^{3+\beta}}{\Gamma \left( 4+\beta \right) }\hspace{-1pt}+\hspace{-1pt}
\tfrac {{x}^{2+\beta}}{\Gamma \left( 3+\beta \right) }.
\end{equation*}
The exact solution is given by
\begin{equation*}
u_{\text{ex}}:=0,\quad y_{\text{ex}}:=e^tx^4(1-x)^4,\ \ J_{\text{ex}}=1.5. 
\end{equation*}
We applied the presented method on this problem with different values of $\beta$. 
The obtained errors $E_n^2(y)$ and $E_n(J)$, for various values of $n=m$, 
are reported in Table~\ref{tab3}.
\begin{table*}[ht]
\centering
\caption{\label{tab3} Obtained norms of errors for various values of $\beta$ and $n=m$. (Example~5).}
\begin{tabular}{cllllllllllll} \hline
&  & \multicolumn{5}{c}{$E_n^2(y)$} &  & \multicolumn{5}{c}{$E_n(J)$} \\ 
\cline{3-7}\cline{9-13}
$n=m$ &  & $\beta=0.1$ & $\beta=0.3$ & $\beta=0.5$ & $\beta=0.7$ & $\beta=0.9$ 
&  & $\beta=0.1$ & $\beta=0.3$ & $\beta=0.5$ & $\beta=0.7$ & $\beta=0.9$ \\ 
\cline{1-1}\cline{3-7}\cline{9-13}
3 &&1.7e-02 &1.4e-02 &1.0e-02 &6.6e-03 &2.7e-03 &&8.4e-04 &5.4e-04 &3.0e-04 &1.3e-04 &2.1e-05 \\ 
4 &&4.9e-03 &3.8e-03 &2.7e-03 &1.7e-03 &6.7e-04 &&4.6e-05 &2.8e-05 &1.5e-05 &6.7e-06 &1.5e-06 \\ 
5 &&5.3e-04 &4.5e-04 &3.7e-04 &2.8e-04 &1.5e-04 &&6.5e-07 &4.8e-07 &3.3e-07 &2.0e-07 &6.4e-08 \\ 
6 &&1.2e-04 &1.2e-04 &1.1e-04 &9.1e-05 &6.0e-05 &&2.9e-08 &2.7e-08 &2.3e-08 &1.7e-08 &8.6e-09 \\ 
7 &&8.5e-08 &8.9e-08 &8.9e-08 &8.7e-08 &8.2e-08 &&8.7e-15 &1.1e-14 &1.3e-14 &1.4e-14 &1.5e-14 \\ 
8 &&7.2e-09 &7.3e-09 &7.4e-09 &8.0e-09 &1.1e-08 &&2.9e-15 &4.4e-15 &6.2e-15 &8.4e-15 &1.2e-14 \\ 
9 &&9.6e-10 &1.3e-09 &2.2e-09 &4.1e-09 &9.4e-09 &&2.7e-15 &4.2e-15 &6.2e-15 &8.1e-15 &9.9e-15 \\ 
10 &&8.5e-10 &1.3e-09 &2.1e-09 &4.1e-09 &9.5e-09 &&2.2e-16 &9.8e-16 &2.7e-15 &6.7e-15 &8.5e-15 \\ 
\hline
\end{tabular}
\end{table*}
As we see, when $\beta$ is close to 1, 
then the accuracy is decreased slightly. This fact is predictable. 
Indeed, if $\beta$ is close to 1, then equation \eqref{main-prob-dyn} 
tends to the advection equation, which is a difficult problem 
from the numerical point of view. However, we see that our method can obtain 
a solution with 8 digits of accuracy, even for large values of $\beta$.


\section{Conclusions}
\label{sec:6}

We proposed a fully direct pseudospectral method 
to solve the optimal control of a two-sided space-fractional diffusion equation. 
Thanks to useful properties of the Jacobi polynomials, accurate and stable 
procedures for deriving the left and right fractional differentiation matrices 
were presented. In our method, the solution of the problem reduces to the solution 
of a convex quadratic programming problem. Five examples were solved 
and results reported. These results show that the fully 
direct pseudospectral method is efficient and provides 
accurate results, whereas a small number of collocation points 
is used and a low CPU time is consumed. Moreover, our third example shows 
that the method here introduced can be also applied with success to difficult 
optimal control problems with inequality 
constraints on the state function. 

Obtaining some theoretical estimates for the approximation 
errors would be desirable. This work is currently in progress.


\begin{acks}
Torres has been partially supported by FCT 
(The Portuguese National Science Foundation) 
through the R\&D unit CIDMA, project UID/MAT/04106/2013. 
Bozorgnia was supported by FCT fellowship SFRH/BPD/33962/2009.
	
The authors are very grateful to three anonymous referees for carefully reading 
their manuscript and for several comments and suggestions which helped them
to improve the paper.
\end{acks}



\end{document}